\documentclass[a4paper, 12pt]{article}
\usepackage{blindtext}
\usepackage[a4paper, inner=1.5cm, outer=1.5cm, top=1.5cm, bottom=1.5cm, bindingoffset=0cm]{geometry}
\usepackage[scaled=92]{helvet}
\usepackage{microtype}
\usepackage{graphicx}
\usepackage{amsmath}
\usepackage{fancyhdr}
\usepackage{index}
\makeindex
\usepackage{setspace}
\usepackage{amsmath}
\usepackage{amsfonts}
\usepackage{amssymb}
\usepackage[utf8]{inputenc}
\usepackage{mathabx}

\usepackage{amscd}
\usepackage{amsthm}
\usepackage{amsxtra}

\usepackage{mathrsfs}
\usepackage{graphics}
\usepackage{tikz}
\usetikzlibrary{matrix}

\usepackage{enumitem}

\usepackage[all,cmtip]{xy}

\theoremstyle{plain}

\newtheorem{theorem}[subsection]{Theorem}
\newtheorem{proposition}[subsection]{Proposition}
\newtheorem{lemma}[subsection]{Lemma}
\newtheorem{corollary}[subsection]{Corollary}

\theoremstyle{definition}

\newtheorem{definition}[subsection]{Definition}

\newtheorem{aforisma}[subsection]{}
\newtheorem{example}[subsection]{Example}
\newtheorem{remark}[subsection]{Remark}

\newcommand{\sig}{\mathbb{S}ig}

\newcommand{\ene}{\textbf{N}}
\newcommand{\N}{\textbf{N}}

\title{On homotopical and cohomological interpretations of Logic}

\author{Thiago Alexandre\footnote{University of S\~{a}o Paulo, Brazil, thiago2.alexandre@usp.br} \\
Gabriel Bittencourt Rios\footnote{University of Toulouse, France, gb.rios01@gmail.com}\\ 
Hugo Luiz Mariano\footnote{University of S\~{a}o Paulo, Brazil, hugomar@ime.usp.br}}

\date{}

\begin{document}

\maketitle

\begin{abstract}
    Building over some ideas of Ren\'e Guitart, we provide a categorical framework  towards some  deviation notions in abstract logic. 
\end{abstract}

\section{Introduction}\label{intro}

 \par Starting from the formal reflections of Ren\'{e} Guitart, we wish to put forth a (co)homological/homotopical reading of the notion of satisfaction in Logic. In \cite{Gui1}, \cite{Gui2} and \cite{Gui3}, Guitart outlines a program aiming to geometrize logic. There, this project goes through several considerations, finally leading Guitart to defend the slogan:
         $$ \text{Logic} = \text{Homological Algebra}. $$
Especifically, using Ehresmann's sketches of \cite{Ehres} and some of Andr\'{e}'s simplicial methods (cf., e.g., \cite{Andre}), Guitart constructs in \cite{Gui3} (co)homology groups $H_{\bullet}^{\varphi}(M), H^{\bullet}_{\varphi}(M)$ for a given sentence $\varphi$ and structure $M$ over some language $\mathcal{L}$. This attribution is such that the structures are acyclic (i.e., the (co)homology groups are trivial) whenever $M\vDash\varphi$. In this sense, we have the notable fact that the (co)homology groups measure the obstruction of a structure $M$ being a model of a sentence $\varphi$.
\par Remarkably, Guitart's technique may be expanded. His specific argument goes uses locally free diagrams and the construction of the sketch associated with a first order language. However, the general argument is strictly categorical: given a small subcategory $A$ of some locally small category $\mathcal{C}$ and a functor $ T: A \longrightarrow \mathcal{A}b $ from $A$ to abelian group\footnote{In practice, we are often interested in composition a given forgetful functor $U: A \to Set$ with the free abelian group one, $\mathbb{Z}[-] : Set \to \mathcal{A}b$}, we assign to each object $X$ of $\mathcal{C}$ the homology and cohomology of the small category $A/X$ with coefficients in $T$: 
         $$ \mathsf{H}_{\ast}(A/X; T), \quad \quad \mathsf{H}^{\ast}(A/X;T). $$ 
Remember that the category $A/X$ is formed by the pairs $(a,s)$ such that $a$ is an object of $A$ and $s: a \rightarrow X$ is an arrow from $a$ to $ X$ in $\mathcal{C}$, and the morphisms of $A/X$ are clearly defined by commutative triangles:

\[
\xymatrix{
   a \ar[dr]_{s} \ar[rr]^{u}   &               &     b \ar[dl]^{t} \\
                              &     X         &
}
\]
Similar to above, or (co)homology groups $H_{\bullet}(A/X;T)$, $H^{\bullet}(A/X;T)$ are such that they are trivial whenever $X$ is an object of $A$. Furthermore, the groups depend only on the \emph{homotopy type} of $A/X$. We note that if we consider the full subcategory $\mathsf{Mod}(\varphi)$ of models inside the category of $\mathcal{L}$-structures $\mathcal{L}_Str$ and the constant functor $k_{\mathbb{Z}}: A \rightarrow \mathcal{A}b$ we can recover Guitart's scheme. Nonetheless, the generality of this setup indicates that the (co)homological interpretation put forth by Guitart does not depend intrisically on the particular semantics of first-order logic. Essentially, to use our method, we need only decide a category to function as structures and a subcategory of the aforementioned to function as models of some sentence. 

A good formalism to provide the ingredients to our method above may be found in the theory of institutions developed by Diaconescu in \cite{Diac}. In short, it proposes an abstract Model Theory where the notion of satisfiability $\models$ is primitive. The scope of the project is quite high indeed, being able treat well not only the ordinary first-order, modal and propositional logics, but also with the more esoteric examples that naturally appear in computer science. Furthermore, the theory of institutions is already written in a categorical language. Therefore, it seems natural to generalize Guitart's arguments to the framework of institutions, and show that his initial slogan, is much deeper than it may seem at first sight, admiting a precise mathematical formalization in the theory of categorical homotopy, so driven by Grothendieck (see \cite{Malts}, \cite{PS}, \cite{Der}).

Finally, we seek to propose a cohomological (but also homotopic) concept of elementary equivalence, which generalizes the traditional definition, and which potentially has relevant consequences in logic (see also \cite{Hend}). Expanding on this point, we aim to introduce (co)homology techniques as a way to remedy the rigid nature of the satisfaction relationship. In fact, satisfaction, à la Tarski, is a binary relationship , not admitting an intermediary notion of ``satisfiability". (Co)homological methods could fill this gap, creating a finer relationship of elementary equivalence. In the presentation, we will illustrate this point through mathematical examples.

{\em Convention:} We fixed two universes $\mathsf{U}$ and $\mathsf{V}$ with $\omega \in \mathsf{U} \in \mathsf{V}$. The universe $\mathsf{U}$ is called the \emph{local universe} and the term set (resp. categories, small categories) will be reserved for $\mathsf{V}$-sets (resp. $\mathsf{V}$-categories, $\mathsf{V}$-small categories). In more details: Let $\mathbf{U}$ be a universe. A category $C$ is called $\mathbf{U}$-\emph{small} when $\mathsf{Ob} A, \mathsf{Fl} A \in \mathbf{U}$, and $\mathbf{U}$-\emph{locally small} when $Hom_{A}(a,b) \in \mathbf{U}$ for all pair $(a,b)$ of objects in $A$. Finally, $A$ is called a $\mathbf{U}$-\emph{category} when it is $\mathbf{U}$-locally small and $\mathsf{Ob} A \subseteq \mathbf{U}$. A set $x$ is $\mathbf{U}$-small when $x \in \mathbf{U}$.

\section{Logic}

The word ``Logic'' for us will the identified with two of its abstract approaches: institution theory and $\pi$-institution theory (\cite{GB}, \cite{Diac}, \cite{FS}, \cite{MaPi}, \cite{RSPM}).

The concept of \emph{institution} was introduced by J. A. Goguen and R. M. Burstall (see \cite{GB}) in order to present a unified mathematical formalism for the notion of a formal logical system, i.e. it provides a \emph{``...categorical abstract model theory which formalizes the intuitive notion of logical system, including syntax, semantic, and satisfaction relation between them...''} (\cite{Diac}). This means that it encompasses the abstract concept of universal model theory for a logic: it contains a satisfaction relation between models and sentences that is ``stable under change of notation''. The are several natural examples of institutions, and a systematic study of abstract model theory based on the general notion of institution is presented in Diaconescu's book \cite{Diac}.

A proof-theoretical variation of the notion  of institution, the concept of \emph{$\pi$-institution}, was introduced by Fiadeiro and Sernadas in  \cite{FS}: it formalizes the notion of a deductive system and  \emph{``...replace the notion of model and satisfaction by a primitive consequence operator (\`a la Tarski)''}. Categories of propositional logics endowed with natural notions of translation morphisms provide examples of $\pi$-institutions. 

\begin{definition}\label{Institution}
  A local institution $\mathbf{I}$ consists of a quadruple
     $$ \mathbf{I} = (\mathscr{S}ig, \mathsf{Fm}, \mathsf{Mod}, \models) $$
     \[\xymatrix{
&\mathscr{S}ig\ar[ld]_{\mathsf{Mod}}\ar[rd]^{\mathsf{Fm}}&\\
(\mathcal{CAT}_{\mathsf{U}})^{op}&\models&\mathcal{E}ns_{\mathsf{U}}
}\]

where
\begin{enumerate}
  \item $\mathscr{S}ig$ is a small category, called the category of signatures. The objects (resp. arrows) of $\mathscr{S}$ will be denoted by the symbols $\Sigma, \Sigma', \Sigma''...$ (resp. $\sigma, \sigma', \sigma''...$).
  \item $\mathsf{Fm}: \mathscr{S}ig \rightarrow \mathcal{E}ns_{\mathsf{U}}$ is a functor from the category of signatures to the category of $\mathsf{U}$-sets \footnote{Local sets.}, which assigns to each signature $\Sigma$ a $\mathsf{U}$-set $\mathsf{Fm}(\Sigma)$ of formulas over $\Sigma$, called $\Sigma$-formulas, and for each morphism of signatures $\sigma: \Sigma \rightarrow \Sigma'$ a function
     $$ \sigma_{\sharp} =_{df} \mathsf{Fm}(\sigma): \mathsf{Fm}(\Sigma) \longrightarrow \mathsf{Fm}(\Sigma'). $$
  \item $\mathsf{Mod}: \mathscr{S}ig^{op} \rightarrow \mathcal{CAT}_{\mathsf{U}}$ is a functor from the dual category of signatures to the category of $\mathsf{U}$-categories, which assigns to each signature $\Sigma$ a $\mathsf{U}$-category $\mathsf{Mod}(\Sigma)$ of $\Sigma$-models, and to each morphism of signatures $\sigma: \Sigma \rightarrow \Sigma'$ a functor
    $$ \sigma^{\sharp}=_{df}\mathsf{Mod}(\sigma): \mathsf{Mod}(\Sigma') \longrightarrow \mathsf{Mod}(\Sigma). $$
  \item $\models$ is a function assigning to each signature $\Sigma$ a relation 
      $$ \models_{\Sigma} \subseteq Ob(\mathsf{Mod}(\Sigma)) \times \mathsf{Fm}(\Sigma) $$
called $\Sigma$-satisfiability. Given a model $\Sigma$-model $M$ and a $\Sigma$-formula $\varphi$, we write $M \models_{\Sigma} \varphi$ to indicate that the pair $(M, \varphi)$ is an element of $\models_{\Sigma}$. 

\end{enumerate}

Moreover, the following conditions are verified: \\
 
 (I1). If $\sigma:\Sigma \rightarrow \Sigma'$ is a morphism of signatures, $M'$ is a $\Sigma'$-model and $\varphi$ is a $\Sigma$-formula, then $M' \models_{\Sigma'} \sigma_{\sharp}(\varphi)$ iff $\sigma^{\sharp}(M') \models_{\Sigma} \varphi$. \\
 
 (I2). If $M$ and $M'$ are two isomorphic $\Sigma$-models, then $M \models_{\Sigma} \varphi$ iff $M' \models_{\Sigma} \varphi$ for every $\Sigma$-formula $\varphi$.

\end{definition}

\begin{definition}
A $\pi$-Institution $J=\langle\sig,Sen,\{C_{\Sigma}\}_{\Sigma\in|\sig|}\rangle$ is a triple with its first two components exactly the same as the first two components of an institution and, for every $\Sigma\in|\sig|$, a closure operator $C_{\Sigma}:\mathcal{P}(Sen(\Sigma))\to\mathcal{P}(Sen(\Sigma))$, such that the following coherence conditions holds, for every $f:\Sigma_{1}\to\Sigma_{2}\in Mor(\sig)$:

\[Sen(f)(C_{\Sigma_{1}}(\Gamma))\subseteq C_{\Sigma_{2}}(Sen(f)(\Gamma)),\ for\ all\ \Gamma\subseteq Sen(\Sigma_{1}).\]

\end{definition}

In \cite{MaPi} and \cite{RSPM} are presented relations between categories whose objects are by institutions (in a certain universe) and categories formed by $\pi$-institutions.

{\em Notation and Terminology}: Given a signature $\Sigma$ and a set $\Gamma$ of $\Sigma$-formulas, we denote by $\overline{\Gamma}$ the set of objects $M$ in $\mathsf{Mod}(\Sigma)$ such that $M \models_{\Sigma} \varphi$ for all $\varphi \in \Gamma$, and by $\mathsf{Mod}_{\Sigma}[\Gamma]$ the full subcategory of $\mathsf{Mod}(\Sigma)$ generated by $\overline{\Gamma}$. In particular, when $\Gamma= \{\varphi\}$ for a $\Sigma$-formula  $\varphi$, we denote just by $\mathsf{Mod}_{\Sigma}[\varphi]$ (resp. $\overline{\varphi}$) the category $\mathsf{Mod}[\Gamma]$ (resp. the set $\overline{\Gamma}$). Given a $\Sigma$-model $M$ and a $\Sigma$-formula $\varphi$, we can always form a projection functor
   $$ \delta_{M,\varphi}: \mathsf{Mod}_{\Sigma}[\varphi]/M \longrightarrow \mathsf{Mod}_{\Sigma}[\varphi], \quad (I,p) \mapsto I. $$
which will be called the \emph{deviation functor} from $M$ to $\varphi$. More generally, given a set $\Gamma$ of $\Sigma$-formulas, we can form the \emph{deviation functor}:
   $$ \delta_{M, \Gamma}: \mathsf{Mod}_{\Sigma}[\Gamma]/M \longrightarrow \mathsf{Mod}_{\Sigma}[\Gamma]. $$
 We also use the notation $M \models_{\Sigma} \Gamma$ to indicate that $M \models_{\Sigma} \varphi$ for every $\varphi \in \Gamma$, and we write $\Gamma \Vdash \varphi$ to indicate that for every $\Sigma$-model $M$ such that $M \models_{\Sigma} \Gamma$, we have $M \models_{\Sigma} \varphi$. Finally, for each $\Sigma$-model $M$, we define the set 
    $$ \overline{M} =_{df} \{\varphi \in \mathsf{Fm}(\Sigma): M \models_{\Sigma} \varphi\},$$
and remark that under the previous notations and definitions, we have
    $$ \overline{\Gamma} = \{M \in Ob(\mathsf{Mod}(\Sigma)): M \models_{\Sigma} \Gamma\} $$
and
   $$ \Gamma \Vdash \varphi \quad \text{iff} \quad \overline{\Gamma} \subseteq \overline{\varphi}. $$
More generally, given a class $\mathcal{M}$ of $\Sigma$-models, we define the set $\overline{\mathcal{M}}$ of $\Sigma$-formulas $\varphi$ such that $M \models_{\Sigma} \varphi$ for all $M \in \mathcal{M}$.

\begin{proposition}
  Let $\mathbf{I} = (\mathscr{S}ig, \mathsf{Fm}, \mathsf{Mod}, \models)$ be an institution.
  \begin{enumerate}
    \item $\Gamma \subseteq \overline{\overline{\Gamma}}$
    \item $\mathcal{M} \subseteq \overline{\overline{\mathcal{M}}}$
    \item If $\Gamma \subseteq \Gamma'$, then $\overline{ \Gamma} \subseteq \overline{\Gamma'}$
    \item If $\mathcal{M} \subseteq \mathcal{M}'$, then $\overline{\mathcal{M}} \subseteq \overline{\mathcal{M}'}$.
    
  \end{enumerate}

\end{proposition}

\begin{definition}\label{Theories, Elementary}
Let $\mathbf{I} = (\mathscr{S}ig, \mathsf{Fm}, \mathsf{Mod}, \models)$ be an institution, $\Sigma$ be a signature of $\mathbf{I}$ and $\Gamma$ (resp. $\mathcal{M}$) be a set (resp. class) of $\Sigma$-formulas (resp. $\Sigma$-models). We say that $\Gamma$ (resp. $\mathcal{M}$) is a theory (resp. is elementary) if $\Gamma = \overline{\overline{\Gamma}}$ (resp. $\mathcal{M} = \overline{\overline{\mathcal{M}}}$).

\end{definition}

A $\Sigma$-theory $\Gamma$ is called \emph{consistent} if $\overline{\Gamma} \neq \varnothing$. Otherwise, we say that $\Gamma$ is \emph{inconsistent}. We call two $\Sigma$-models $M$ and $M'$ elementary equivalent, and write $M \equiv M'$ if $\overline{M}= \overline{M'}$. It follows from (I2) of (\ref{Institution}) that every two isomorphic $\Sigma$-models are elementary equivalent, but the reciprocal is not necessary the case. \\

The notation of algebraic geometry is very suggestive in model theory. It is possible to think about the signatures as commutative rings with unity, the $\Sigma$-formulas as polynomials with coefficients in $\Sigma$ and the $\Sigma$-models as points of the $\Sigma$-space, which is the set $Spec(\Sigma)$ of objects in the category $\mathsf{Mod}(\Sigma)$. For each $\Sigma$-formula $\varphi$, we can then define a function
     $$ [\varphi]: Spec(\Sigma) \longrightarrow \{0,1\} $$
 such that $[\varphi]_{M}=0$ if $M \models_{\Sigma} \varphi$, and $[\varphi]_{M}=1$ otherwise.    
We could define the symbols
    $$ V(\varphi)=_{df} \{M \in Spec(\Sigma): [\varphi]_{M}=0\} $$
    $$ D(\varphi)= Spec(\Sigma) - V(\varphi)$$
    $$ V(\Gamma) = \bigcap_{\varphi \in \Gamma} V(\varphi) $$
    $$ D(\Gamma)= \bigcup_{\varphi \in \Gamma} D(\varphi) $$
    $$ \mathcal{O}_{\Sigma,M}= \{\varphi \in \mathsf{Fm}(\Sigma): [\varphi]_{M}=0\} $$
        $$ I(\mathcal{U})= \bigcap_{M \in \mathcal{U}} \mathcal{O}_{\Sigma, M}$$
With the above notations, we have
   $$ \mathcal{U} \subseteq V(I(\mathcal{U})), \quad \Gamma \subseteq I(V(\Gamma)). $$

We think $Spec(\Sigma)$ as a space in which points are the $\Sigma$-models, and for a set of $\Sigma$-sentences $\Gamma$, we think $V(\Gamma)$ as a (closed) subspace of $Spec(\Sigma)$, formed by the models of the sentences of $\Gamma$. Then, we can say that $\Gamma$ is compact if for every $\Sigma$-sentence $\varphi$ such that $\Gamma \Vdash \varphi$, there exists a finite subset $\Gamma_{0} \subseteq \Gamma$ such that $\Gamma_{0} \Vdash \varphi$. \\

\begin{definition}\label{Syntatic consequence relation}
  A syntactic consequence relation $\vdash$ for an institution $\mathbf{I}$ is a map assigning to each signature $\Sigma$ of $\mathbf{I}$ a relation $\vdash_{\Sigma}$ between sets of $\Sigma$-formulas $\Gamma$ and $\Sigma$-formulas $\varphi$:
       $$ \Gamma \vdash_{\Sigma} \varphi. $$
We say that $\vdash$ is sound if the relation $\Gamma \vdash_{\Sigma} \varphi$ always implies the relation $\Gamma \Vdash \varphi$, and we say that $\vdash$ is complete if $\vdash$ is sound and the relation $\Gamma \Vdash \varphi$ always implies $\Gamma \vdash_{\Sigma} \varphi$.

\end{definition}

Therefore, if $\vdash$ is a complete syntactic consequence relation on an institution $\mathbf{I}$, to say that $\Gamma \vdash_{\Sigma} \varphi$ is the same to say that $\Gamma \Vdash \varphi$, which means that we can think the relation $\Vdash$ as some sort of syntactic consequence relation from semantics. In fact, if $\Gamma$ is a theory, then $\varphi \in \Gamma$ iff $\Gamma \Vdash \varphi$. \\

Inspired in the previous reasoning, we can formulate the notion of \emph{independence} semantically. Indeed, let $\Gamma$ be a theory in signature $\Sigma$ and $\varphi$ be a $\Sigma$-sentence. If $\varphi$ is \emph{refutable} from $\Gamma$, then $V(\Gamma) \cap V(\varphi)=\varnothing$, otherwise, $\varphi$ is \emph{irrefutable} from $\Gamma$, which means that $V(\Gamma) \cap V(\varphi) \neq \varnothing$. We say that $\varphi$ is \emph{provable} from $\Gamma$ when $\Gamma \Vdash \varphi$, which means that $V(\Gamma) \subseteq V(\varphi)$. Otherwise, we say that $\varphi$ is \emph{unprovable} from $\Gamma$, which means that there exists at least one $\Sigma$-model $M$ in $V(\Gamma)$ which is not in $V(\varphi)$, i.e., $V(\Gamma) \cap D(\varphi) \neq \varnothing$. If $\varphi$ is both \emph{irrefutable} and \emph{unprovable} from $\Gamma$, then we say that $\varphi$ is \emph{independent} from $\Gamma$:
$$ \text{(Irrefutable)} \quad V(\Gamma) \cap V(\varphi) \neq \varnothing, \quad \quad \text{(Unprovable)} \quad V(\Gamma) \cap D(\varphi) \neq \varnothing .$$

Now, suppose that $M$ is a $\Sigma$-model for the theory $\mathbb{T}$, i.e., $M \in V(\mathbb{T})$, which means that $M \models_{\Sigma} \phi$ for every $\phi \in \mathbb{T}$. Let $\varphi$ be an arbitrary $\Sigma$-sentence. In order to know if $\varphi$ is an element of $\mathbb{T}$, which means that $\mathbb{T} \Vdash \varphi$ (since $\mathbb{T}$ is a theory). we have to behold the projection functor

\begin{aforisma}
  Let $\mathbf{I}$ be an institution, $\Sigma$ be a signature of $\mathbf{I}$, $\mathbb{T}$ be a $\Sigma$-theory and $T$ be a $\Sigma$-model. A $\mathbb{T}$-diagram for $T$ (if it exists) is given the following data: \\

(i). A small category $A$ \\

(ii). A functor $D: A \rightarrow \mathsf{Mod}[\mathbb{T}]$ \\

(iii). An inductive cone $d=(d_{a}: D_{a} \rightarrow T)_{a \in Ob(A)}$ such that, for every $\Sigma$-model $M$ of the $\Sigma$-theory $\mathbb{T}$, there exists a canonical isomorphism:
     $$ Hom_{\mathsf{Mod}(\Sigma)}(T,M) \cong \varprojlim_{a \in Ob(A)} Hom_{\mathsf{Mod}[\mathbb{T}]}(D_{a},M) $$
induced from the arrows
     $$ Hom_{\mathsf{Mod}(\Sigma)}(T,M) \longrightarrow Hom_{\mathsf{Mod}[\mathbb{T}]}(D_{a},M), \quad f \mapsto f \circ d_{a} $$
for $a \in Ob(A)$. \\

In particular, if the category $\mathsf{Mod}(\Sigma)$ admits small inductive limits, then there exists an isomorphism 
     $$  \varinjlim_{a \in Ob(A)} D_{a} \longrightarrow T $$
in $\mathsf{Mod}(\Sigma)$, which means that $T$ is an inductive limit of models of the $\Sigma$-theory $\mathbb{T}$. Clearly, if $\mathsf{Mod}[\mathbb{T}]$ is closed under small inductive limits and there exists a $\mathbb{T}$-diagram for $T$, then $T$ is an object $\mathsf{Mod}[\mathbb{T}]$. Yet, we do not suppose that $\mathsf{Mod}(\Sigma)$ admits small inductive limits and neither that $\mathsf{Mod}[\mathbb{T}]$ is closed under small inductive limits. \\

From the existence of a $\mathbb{T}$-diagram for $T$, we can deduce a functor
      $$ \Omega: A \longrightarrow \mathsf{Mod}[\mathbb{T}]/T $$
such that,
      $$ \Omega_{a} = (D_{a}, D_{a} \xrightarrow{d_{a}} T), \quad \quad a \in Ob(A) $$
and for an arrow $\varphi: a' \rightarrow a$ in $A$, $\Omega_{\varphi}$ is the arrow in $\mathsf{Mod}[\mathbb{T}]/T$ defined by the commutative triangle:
\[
\xymatrix{
   D_{a} \ar[rr]^{D(\varphi)} \ar[dr]_{d_{a}}  &                &   D_{a'} \ar[dl]^{d_{a'}} \\
                                              &    T           &
}
\]
The functor $\Omega$ satisfies the following property: for every object $(M,p)$ of $\mathsf{Mod}[\mathbb{T}]/M$, where $p: M \rightarrow T$ is an arrow of $\Sigma$-models, the category $A/(M,p)$ is formed by the pairs $(a,f)$ where $a \in Ob(A)$ and $f: D_{a} \rightarrow M$ is an arrow of $\Sigma$-models of $\mathbb{T}$ such that $p \circ f = d_{a}$.

\end{aforisma}

We can exchange $\mathsf{Mod}[\mathbb{T}]/T$ by a small category of indices $I_{\Sigma}(\mathbb{T},T)$, and consider the category  \\

In the sequel, we give several main examples of institutions in order to illustrate to the reader the generality of the concept (\ref{Institution}).

\begin{example}\label{PL}
 We recall briefly the language of propositional logic and show how classical propositional logic is an example of an institution. First, we consider the set $\omega$ of finite ordinals, the set $C=\{\implies, \neg, \vee, \wedge\}$ of logical connective symbols and the set $L=\{[,]\}$ of linguistic symbols. Then, we define the set $F$ of formulas of propositional logic in the usual way by recursion. The finite subsets of $\omega$ form a partially ordered set by inclusion, and hence, they also define a category $\mathscr{S}$, which will be the category of signatures. For each signature $\Sigma$, let $\mathsf{Fm}(\Sigma)$ be the set of formulas $\varphi \in F$ such that $FV(\varphi) \subseteq \Sigma$, where $FV(\varphi)$ is the usual set of free variables of $\varphi$, also defined by recursion \footnote{We define $FV(p)=p$ for an atomic formula $p$, which is just an element of $\omega$, $FV(\neg \varphi) = FV(\varphi)$ and $FV(\varphi \lozenge \psi)= FV(\varphi) \cup FV(\psi)$ for $\lozenge \in \{\implies, \vee, \wedge\}$}, and let $\mathsf{Mod}(\Sigma)$ be the discrete category generated by the set $2^{\Sigma}$ of functions from $\Sigma$ to $\{0,1\}$. If $\sigma: \Sigma \rightarrow \Sigma'$ is a morphism of signatures, then $\Sigma \subseteq \Sigma'$, and hence, it defines a morphism $\sigma_{\sharp}:\mathsf{Fm}(\Sigma) \rightarrow \mathsf{Fm}(\Sigma')$ given by the inclusion $\mathsf{Fm}(\Sigma) \subseteq \mathsf{Fm}(\Sigma')$. Moreover, we can also define a unique morphism $\sigma^{\sharp}: 2^{\Sigma'} \rightarrow 2^{\Sigma}$ restricting each function of the form $f: \Sigma' \rightarrow \{0,1\}$ to $\Sigma$, i.e., $\sigma^{\sharp}(f)= f \vert_{\Sigma}$. Finally, we define for each signature $\Sigma$ a relation $f \models_{\Sigma} \varphi$, where $f \in Ob(\mathsf{Mod}(\Sigma))$ and $\varphi \in \mathsf{Fm}(\varphi)$, as following: if $\varphi \in \Sigma$, $f \models_{\Sigma} \varphi$ iff $f(\varphi)=1$; if $\varphi$ is $\varphi$. 
  
\end{example}

\begin{example}\label{FOL}
 The semantics of each first-order logic $L_{\kappa, \lambda}$ can be organized in a institution.
 
Let $Lang$ denote the  category of  languages $L = ((F_n)_{n \in \N},(R_n)_{n \in \N})$, -- where  $F_n$ is a set of symbols of $n$-ary function symbols and $R_n$ is a   set of symbols of $n$-ary relation symbols, $n \geq 0$ -- and language morphisms\footnote{That can be chosen ``strict" (i.e., $F_n\mapsto F_n'$, $R_n \mapsto R'_n$) or chosen be ``flexible" (i.e., $F_n\mapsto \{n-ary-terms(L')\}$, $R_n \mapsto \{n-ary-atomic-formulas(L')\}$).}. For each pair of cardinals $\aleph_0 \leq \kappa, \lambda \leq \infty$, the category $Lang$  endowed with the usual notion of  $L_{\kappa,\lambda}$-sentences (=  $L_{\kappa,\lambda}$-formulas with no free variable), with the usual association of category of structures and  with the usual (tarskian) notion of satisfaction, gives rise to an institution $I({\kappa,\lambda})$.

\end{example}

\begin{example}\label{Sketches}
The models of Ehresmann's sketches ({\em esquisses}) can be organized in a institution.

\end{example}

\begin{example}\label{Linear Algebra}
  Let $Comm$ be the category of commutative rings with unity and $? \text{-} Mod: Comm^{o} \rightarrow \mathcal{CAT}$ be the functor which associates to each object $\Lambda$ of $Comm$ its category $\Lambda \text{-} Mod$ of $\Lambda$-modules, and to each morphism $\sigma: \Lambda \rightarrow \Lambda'$ of $Comm$ the $\Lambda'$-linear morphism
      $$ \sigma^{\sharp}:\Lambda' \text{-} Mod \longrightarrow \Lambda \text{-} Mod, \quad x \mapsto $$ 
Denote by $Lin(\Lambda)$ the set of systems of linear equations with coefficients in $\Lambda$. Well understood, a prototypical element of $Lin(\Lambda)$ is a subset $L$ of $\Lambda[X_{1},...,X_{n}]$ (for some $n \in \omega$) which contains only polynomials of the form
       $$ a_{1}X_{1}+...+a_{n}X_{n}, \quad \quad a_{i} \in \Lambda. $$
      
With the previous notations, the quadruple
     $$ (Comm, Lin, ? \text{-} Mod, \models ) $$
is an institution, called \emph{linear algebra}.

\end{example}

\section{Cohomology}

\subsection*{Introduction}



After the introduction of derived categories and derived functors by Grothendieck and Verdier, followed by the invention of model categories and homotopical algebra due to Quillen 
, mathematicians started to have a deep understanding of cohomological phenomena. First, replacing the traditional cohomology groups $H^{n}(X;\mathscr{F})$ of a geometric object $X$, with coefficients in a sheaf of abelian groups $\mathscr{F}$, with the total cohomology $\mathsf{R}\Gamma(X; \mathscr{F})$, and then, defining $\mathsf{R}\Gamma(X; \mathscr{F})$ via a \emph{formalism of operations} (in a complete picture, a \emph{Grothendieck's formalism of six operations}). \\

If $F: \mathcal{A} \rightarrow \mathcal{A}'$ is an additive functor between Grothendieck abelian categories, not necessarily exact, then there exists a \emph{total right derived functor}:
     $$ \mathsf{R}F: \mathcal{A} \longrightarrow \mathcal{A}'. $$
The classic $n$-th derived functors $\mathsf{R}^{n}F$, constructed from injective (or projective) resolutions, can be recovered from $\mathsf{R}F$ by the formula:
   $$ \mathsf{R}^{n}F(X) = H^{n}(\mathsf{R}F(X)). $$
For example, if $\mathcal{E}$ is a topos, then we can form the topos $\mathcal{E}_{ab}$ of abelian sheaves in $\mathcal{E}$. An object $G$ of $\mathcal{E}_{ab}$ is a sheaf in $\mathcal{E}$, endowed with two morphisms
    $$ \mu: G \times G \rightarrow G, \quad \quad e: \ast \rightarrow G $$
such that, for each object $X$ of $\mathcal{E}$, the set $Hom_{\mathcal{E}}(X,G)$ is an abelian group, where the sum 
    $$ +: Hom_{\mathcal{E}}(X,G) \times Hom_{\mathcal{E}}(X,G) \longrightarrow Hom_{\mathcal{E}}(X,G) $$ 
 is given by the composition of the canonical function:
    $$ Hom_{\mathcal{E}}(X,G) \times Hom_{\mathcal{E}}(X,G) \longrightarrow Hom_{\mathcal{E}}(X, G \times G), \quad (f,g) \mapsto f \times g $$
(induced from the universal property of the product $G \times G$) with the function
     $$ Hom_{\mathcal{E}}(X, G \times G) \longrightarrow Hom_{\mathcal{E}}(X,G), \quad f \mapsto \mu \circ f, $$
and the null element $0: X \rightarrow G$ is given by the composition 
  $$ X \xrightarrow{p_{X}} \ast \xrightarrow{e} G. $$
Then, we can define the functor of global sections:
   $$ \Gamma(\mathcal{E}; ?): \mathcal{E}_{ab} \longrightarrow \mathcal{A}b, \quad \mathscr{F} \mapsto \Gamma(\mathcal{E}; \mathscr{F})=_{df} Hom_{\mathcal{E}}(\ast, \mathscr{F}). $$
In particular, if $\mathcal{E}$ is the topos of sheaves over a topological space $X$, then we have
    $$ Hom_{\mathcal{E}}(\ast, \mathscr{F}) \cong Hom_{\mathcal{E}}(h_{X}, \mathscr{F}) \cong \mathscr{F}(X). $$
Therefore, for every sheaf $\mathscr{F}$ of abelian groups over a topological space $X$, we can define the cohomology $\mathsf{H}^{\ast}(X;\mathscr{F})$ by the formula
     $$ \mathsf{H}^{\ast}(X;\mathscr{F})=_{df} \mathsf{R}\Gamma(X; \mathscr{F}) = \mathsf{R}(\pi_{X})_{\ast}(\mathscr{F}), $$
and the $n$-th cohomology groups by the formula
    $$ H^{n}(X;\mathscr{F})=_{df} \mathsf{R}^{n}\Gamma(X;\mathscr{F})= H^{n}(\mathsf{R}\Gamma(X;\mathscr{F})) = H^{n}(\mathsf{R}(\pi_{X})_{\ast}(\mathscr{F})). $$
If $\Lambda$ is an abelian group, regarded as a complex concentrated in zero degree in $\mathsf{C}(\mathcal{A}b)$, then 
    $$ H^{n}(X; \Lambda)=_{df} H^{n}(\mathsf{R}(\pi_{X})_{\ast}(\Lambda_{X})) = H^{n}(\mathsf{R}(\pi_{X})_{\ast}(\pi_{X})^{\ast}(\Lambda)).$$
If the topological space $X$ is locally contractible, then these cohomology groups $H^{n}(X;\Lambda)$ are naturally isomorphic to the singular cohomology groups $H^{n}_{Sing}(X;\Lambda)$ with coefficients in $\Lambda$. \\

Now, modelizing spaces by small categories rather than topological spaces, the topos $Sh(X)$ is replaced by the classifying topos $\widehat{A}$ of presheaves over a small category $A$, and cohomology in this case is translated by the formulas:
   $$ \mathsf{H}^{\ast}(A; \Lambda) = \mathsf{R}(p_{A})_{\ast}(\Lambda_{A}) = \mathsf{R}(p_{A})_{\ast}(p_{A})^{\ast}(\Lambda), \quad \quad H^{n}(A;\Lambda)= H^{n}(\mathsf{H}^{\ast}(A; \Lambda)),$$
for every abelian group $\Lambda$. The connection of spaces as small categories with spaces as topological spaces is given via a natural geometric morphism of topos
    $$ \mu_{A}: Sh(|\ene A|) \longrightarrow \widehat{A} $$
due to Moerdjik, which has the property that, for every locally constant `sheaf' of abelian groups $\mathscr{F}$ in $\widehat{A}$, the canonical morphism
     \[ \mathsf{R}\Gamma(|\ene A|; (\mu_{A})^{\ast}(\mathscr{F})) \longrightarrow \mathsf{R}\Gamma(A; \mathscr{F}) \]
induced from the $1 \rightarrow (\mu_{A})_{\ast}(\mu_{A})^{\ast}$ via the commutative triangle of geometric morphisms
\[
\xymatrix{
  Sh(|\ene A|) \ar[rr]^{\mu_{A}} \ar[dr]_{\widetilde{\pi}_{|\ene A|}}  &              &  \widehat{A} \ar[dl]^{\widehat{p}_{A}} \\
                                                                          &  \mathcal{E}ns &
}
\]
(since $\widetilde{\ast} \simeq \mathcal{E}ns \simeq \widehat{e}$), is an isomorphism, which means that
     $$ \mathsf{H}^{\ast}(A; \Lambda) \cong \mathsf{H}^{\ast}(|\ene A|; \Lambda), $$
and
     $$ H^{n}(A; \Lambda) \cong H^{n}(|\ene A|; \Lambda), \quad \quad n \in \mathbb{Z} $$
for every abelian group $\Lambda$. Moreover, it follows from fact that $|\ene A|$ is always a locally contractible topological space, that $H^{n}(A;\Lambda)$ is isomorphic to the $n$-th singular cohomology group, with coefficients in $\Lambda$, of the topological space $|\ene A|$. \\

Hence, the general form of cohomology of small categories can be formally presented in the following way:  \\

(i). Assigning to each small category $A$ a derived category $\mathbb{D}(A)$, which is the derived category of the abelian category $\widehat{A}_{ab}$ \\

(ii). Assigning to each morphism of small categories $u: A \rightarrow B$ a morphism... \\


\subsection*{Simplicial methods to compute (co)homology}


Cohomology of small categories can be computed using simplicial methods developed by Andr\'{e} in \cite{Andre}. If $T: S \rightarrow \mathcal{A}$ is a functor from a small category $S$ to a Grothendieck abelian category $\mathcal{A}$, and there exists an inclusion functor
   $$ \jmath: S \longrightarrow \mathcal{E} $$
from $S$ to a locally small category $\mathcal{E}$, then we can form the small category $S/X$, which is the evaluation of the functor
    $$ \mathcal{E} \xrightarrow{\jmath^{\ast}} \widehat{S} \xrightarrow{i_{S}}  \mathcal{C}at $$
at the object $X$. Then, we can form the nerve 
$$ \mathbf{N}(S/X)=\mathbf{N}i_{C}\jmath^{\ast}(X) $$
and, for every $n \in \omega$, the set $S_{n}(X)=(\mathbf{N}(S/X))_{n}$. Since $\mathcal{A}$ admits small coproducts, there exists the object
           $$ C_{n}(X;T) = \sum_{\alpha \in S_{n}(X)} T[\alpha] $$
in $\mathcal{A}$, where, for $\alpha \in S_{n}(X)$, 
           $$ T[\alpha] =_{df} T(dom(\alpha_{n})) $$
with $\alpha$ denoting
$$ s_{n} \xrightarrow{\alpha_{n}} s_{n-1} \rightarrow ... \rightarrow s_{1} \xrightarrow{\alpha_{1}} s_{0}  \xrightarrow{\alpha_{0}} X. $$
We can say that $S_{n}(X)$ is the set of $S$-resolutions of dimension $n$ of $X$. 

For $0 \leq i \leq n$, let
   $$ \varepsilon_{0}(s_{n} \xrightarrow{\alpha_{n}}  s_{n-1} \rightarrow ... \rightarrow s_{1} \xrightarrow{\alpha_{n}} s_{0} \xrightarrow{\alpha_{0}} X) = (s_{n} \xrightarrow{\alpha_{n}} s_{n-1} \rightarrow ... s_{1} \xrightarrow{\alpha_{0} \alpha_{1}} X) $$
   $$ \varepsilon_{i}(s_{n} \rightarrow ... \rightarrow s_{i+1} \xrightarrow{\alpha_{i+i}} s_{i} \xrightarrow{\alpha_{i}}\cdots\rightarrow x) = (s_{n} \rightarrow ... \rightarrow s_{i+1} \xrightarrow{\alpha_{i+1}\alpha_{i}} s_{i-1} \cdots\rightarrow X) $$
   $$ \varepsilon_{n}(s_{n} \xrightarrow{\alpha_{n}}  s_{n-1} \rightarrow ... \rightarrow s_{1} \xrightarrow{\alpha_{n}} s_{0} \xrightarrow{\alpha_{0}} X) = (s_{n-1} \xrightarrow{\alpha_{n-1}} ... \rightarrow s_{0} \xrightarrow{\alpha_{0}} X) $$
the second case only defined for $0 <i <n$. Then, each $\varepsilon_{i}: S_{n}(X) \rightarrow S_{n-1}(X)$ induces a morphism
     $$ d_{i}: C_{n}(X) \longrightarrow C_{n-1}(X) $$
from where we deduce the arrow
      $$ \partial_{n} = \sum_{i=0}^{n} (-1)^{i} d_{i}: C_{n}(X) \longrightarrow C_{n-1}(X). $$
We make more explicit the previous arrows. For each $\alpha \in S_{n}(X)$, with
   $$ \alpha = (\alpha_{n},...,\alpha_{1},\alpha_{0})= (s_{n} \xrightarrow{\alpha_{n}} s_{n-1} \rightarrow ... \rightarrow s_{1} \xrightarrow{\alpha_{1}} s_{0} \xrightarrow{\alpha_{0}} X) $$
 we can form the morphism
      $$ \langle \alpha_{n},..., \alpha_{1}, \alpha_{0} \rangle: T[\alpha_{n},...,\alpha_{1}, \alpha_{0}] \longrightarrow  C_{n}(X;T). $$
In order to define an arrow $d_{i}: C_{n}(X;T) \rightarrow C_{n-1}(X;T)$ in $\mathcal{A}$, it is enough to define $d_{i} \langle \alpha_{n},...,\alpha_{1}, \alpha_{0} \rangle$ for each $\alpha \in S_{n}(X)$. 
      
\vspace{1cm}

Dually, the deviation of $X$ be in $S$ can also be represented by homologies given by: 
$$\beta : X \to M_0 \to M_1\to \cdots \to M_{n-1}\to M_{n}$$

\subsection*{Homotopy and cohomology of spaces}

\begin{aforisma}
  Whenever $A$ is a $\mathbf{U}$-small category, we denote by $\widehat{A}$ the usual category of functors over $A$ with values in the category of $\mathbf{U}$-small sets and functions between them. The category $\widehat{A}$, which is called the category of \emph{presheaves over} $A$, can also be thought just in categorical terms as following: the objects of $\widehat{A}$ are functors $p: X \rightarrow A$ for which the domains are $\mathbf{U}$-small categories, and for every object $x \in \mathsf{Ob} X$, the evident functor 
  $$X/x \longrightarrow A/p(x), \quad (x',\alpha) \mapsto (p(x'), p(\alpha)) $$
is an \emph{isomorphism}. A morphism $f: (X,p) \rightarrow (Y,q)$ of $\widehat{A}$ is a functor $f: X \rightarrow Y$ such that the diagram
\[
\xymatrix{
 X \ar[dr]_{p} \ar[rr]^{f}  &   &  Y \ar[dl]^{q} \\
                           & A &
}
\]
commutes. Clearly, the category $\widehat{e}$ of presheaves over the \emph{point} $e$ (the category with just one object and its identity arrow) corresponds to the category of $\mathbf{U}$-small sets, that will be denoted by $\mathcal{P}_{\mathbf{U}}$.

\end{aforisma}

\begin{definition}\label{Fundamental localizers}
 Let $\mathbf{U}$ be a universe and $\mathcal{C}at_{\mathbf{U}}$ be the category of $\mathbf{U}$-small categories, i.e., categories $A$ such that $\mathsf{Ob} A, \mathsf{Fl} A \in \mathbf{U}$. A set $\mathcal{W}$ of arrows in $\mathcal{C}at_{\mathbf{U}}$ is called \emph{weakly saturated} when it satisfies the following properties:
 
 WS1 All the identity arrows are in $\mathcal{W}$
 
 WS2 If two between three arrows in any commutative triangle are in $\mathcal{W}$, then so is the third one.
 
 WS3 If $i: A \rightarrow X$ and $r: X \rightarrow A$ are two morphisms in $\mathcal{C}at$ such that $ri = 1_{A}$ and $ir \in \mathcal{W}$, then $r \in \mathcal{W}$.
 
A $\mathbf{U}$-\emph{fundamental localizer} $\mathcal{W}$ is a set of arrows in $\mathcal{C}at_{\mathbf{U}}$ satisfying the following conditions:

L1 $\mathcal{W}$ is weakly saturated
 
L2 If $A$ admits a terminal object $e$, then the canonical functor $p_{A}: A \rightarrow e$ from $A$ to the point category $e$ is in $\mathcal{W}$.

L3 If
   \[
   \xymatrix{
     A \ar[dr]_{p} \ar[rr]^{u}   &        &  B \ar[dl]^{q} \\
                                 &  C     ¨&
   }
   \]
is a commutative triangle in $\mathcal{C}at_{\mathbf{U}}$ and for each object $c \in \mathsf{Ob} C$, the functor $u/c: A/c \rightarrow B/c$  is in $\mathcal{W}$, then $u$ is in $\mathcal{W}$.

\end{definition}

\emph{Terminology} - Let $\mathcal{W}$ be a $\mathbf{U}$-fundamental localizer. The elements of $\mathcal{W}$ are called $\mathcal{W}$-equivalences. A $\mathbf{U}$-small category $A$ is called $\mathcal{W}$-\emph{aspherical} if the morphism $p_{A}: A \rightarrow e$ is a $\mathcal{W}$-equivalence. An arrow $u:A \rightarrow B$ of $\mathcal{C}at_{\mathbf{U}}$ is called $\mathcal{W}$-aspherical if for every $b \in \mathsf{Ob} B$, the morphism $u/b: A/b \rightarrow B/b$ is a $\mathcal{W}$-equivalence, and it is called a universal $\mathcal{W}$-equivalence when it is stable under base change, i.e., for every Cartesian square
\[
\xymatrix{
  A' \ar[r]^{v'} \ar[d]_{u'}   &   A \ar[d]^{u} \\
  B' \ar[r]_{v}                &   B
}
\]
the arrow $u': A' \rightarrow B'$ is a $\mathcal{W}$-equivalence.

\begin{definition}
  Fixing a universe $\mathbf{U}$, we define the category of $\mathbf{U}$-homotopy types as being the category:
    $$ \mathsf{Hot}_{\mathbf{U}} =_{df}  (\mathcal{W}_{\infty})^{-1}\mathcal{C}at. $$
  
\end{definition}

\begin{definition}\label{Homotopy invariant}
  A homotopy invariant is a functor
      $$ \mathsf{T}: \mathcal{C}at_{\mathbf{U}} \longrightarrow \mathcal{A} $$
sending $\mathcal{W}$-equivalences into isomorphisms for any $\mathbf{U}$-fundamental localizer $\mathcal{W}$.

\end{definition}

There is a plethora of homotopy invariants appearing in mathematical literature. In this work we focus in cohomology. In what follows, we fix one and for all two universes $\mathbf{U}$ and $\mathbf{V}$ such that $\omega \in \mathbf{U} \in \mathbf{V}$. 

\begin{aforisma}
 Let $\mathcal{A}b_{\mathbf{U}}$ be the $\mathbf{U}$-category of $\mathbf{U}$-small abelian groups, and $\mathsf{C}^{\ast}(\mathcal{A}b_{\mathbf{U}})$ be the category of $\mathbf{U}$-abelian group complexes in cohomological notation. An object of $\mathsf{C}^{\ast}(\mathcal{A}b_{\mathbf{U}})$ is a diagram of abelian groups f the form
 \[
 \xymatrix{
   ..... \ar[r]  &  C^{n} \ar[r]^{d^{n}}  &  C^{n+1} \ar[r]^{d^{n+1}}   & ....
 }
 \]
in $\mathcal{A}b_{\mathbf{U}}$, such that $d^{n+1} \circ d^{n} =0$. More precisely, we can assign to each $\mathbf{U}$-small category $I$, the full subcategory $\mathsf{C}^{\ast}(I, \mathcal{A}b_{\mathbf{U}})$ of the category of functors $F$ from $I$ to $\mathcal{A}b_{\mathbf{U}}$ such that, $F(\beta) \circ F(\alpha) = 0$ for every pair $(\alpha, \beta)$ of composable arrows in $I$, i.e., for which $\beta \circ \alpha$ is defined. Then, choosing the category $I_{\mathbb{Z}}$ generated by the graph
  $$ .... \rightarrow -1 \rightarrow 0 \rightarrow 1 \rightarrow ... \rightarrow n \rightarrow ... $$
one can define $\mathsf{C}^{\ast}(\mathcal{A}b_{\mathbf{U}}) =_{df} \mathsf{C}^{\ast}(I_{\mathbb{Z}}, \mathcal{A}b_{\mathbf{U}})$. If $C^{\bullet}$ is an object of $\mathsf{C}^{\ast}(\mathcal{A}b_{\mathbf{U}})$, then $im (d^{n}) \subset ker (d^{n+1})$, and we can form the abelian groups
         $$ H^{n}(C^{\bullet}) =_{df} ker (d^{n+1})/im (d^{n}), \quad \quad n \in \mathbb{Z}, $$
called the $n$-th cohomology groups of $C^{\bullet}$. If $f^{\bullet}: C^{\bullet} \rightarrow D^{\bullet}$ is a morphism of $\mathsf{C}^{\ast}(\mathcal{A}b_{\mathbf{U}})$, then it induces, for each $n \in\mathbb{Z}$, a canonical group homomorhism
         $$ H^{n}(f^{\bullet}): H^{n}(C^{\bullet}) \longrightarrow H^{n}(D^{\bullet}) $$
between the respective $n$-th cohomology groups. An arrow $f^{\bullet}$ of $\mathsf{C}^{\ast}(\mathcal{A}b_{\mathbf{U}})$ is a quasi-isomorphism when $H^{n}(f^{\bullet})$ is an isomorphism of $\mathbf{U}$-abelian groups for all $n \in \mathbb{Z}$. We denote by $\mathcal{Q}$ the class of quas-isomorphisms in $\mathsf{C}^{\ast}(\mathcal{A}b_{\mathbf{U}})$. The derived category of $\mathbf{U}$-abelian groups is defined as being the localization of $\mathsf{C}^{\ast}(\mathcal{A}b_{\mathbf{U}})$ with respect to quasi-isomorphisms:
   $$ \mathsf{D}(\mathcal{A}b_{\mathbf{U}}) =_{df} \mathcal{Q}^{-1} \mathsf{C}^{\ast}(\mathcal{A}b_{\mathbf{U}}). $$
Hence, the category $\mathsf{D}(\mathcal{A}b_{\mathbf{U}})$ can be characterized by the following universal property: there is a functor
    $$ \gamma: \mathsf{C}^{\ast}(\mathcal{A}b_{\mathbf{U}}) \longrightarrow \mathsf{D}(\mathcal{A}b_{\mathbf{U}}) $$
which is the identity on the objects, sending all quasi-isomorphisms into isomorphisms, and for any functor $F: \mathsf{C}^{\ast}(\mathcal{A}b_{\mathbf{U}}) \rightarrow \mathcal{M}$ satisfying the same condition, there exists a unique functor $\overline{F}: \mathsf{D}^{\ast}(\mathcal{A}b_{\mathbf{U}}) \rightarrow \mathcal{M}$ such that $F = \overline{F} \circ \gamma$.

\end{aforisma}

\begin{aforisma}
For each $\mathbf{U}$-small category $A$, one can form the category $\mathsf{Comp}(A)$ of presheaves over $A$ with values in the category $\mathsf{C}^{\ast}(\mathcal{A}b_{\mathbf{U}})$, and defining the set $\mathcal{Q}_{A}$ of arrows $f: X \rightarrow Y$  in $\mathsf{Comp}(A)$ such that $f_{a}: X_{a} \rightarrow Y_{a}$ is a quasi-isomorphism for every object $a \in \mathsf{Ob} A$, there is a category:
  $$ \mathbf{D}(A) =_{df} (\mathcal{Q}_{A})^{-1} \mathsf{Comp}(A). $$
Morphisms $u: A \rightarrow B$ between $\mathbf{U}$-small categories induce functors defined by composition:
  $$ u^{\ast}: \mathsf{Comp}(B) \longrightarrow \mathsf{Comp}(A), \quad F \mapsto F \circ u, $$
which in turn, admit right adjoints
   $$ u_{\ast}: \mathsf{Comp}(A) \longrightarrow \mathsf{Comp}(B) $$
descending to adjunctions:
    $$ u^{\ast}: \mathbf{D}(B) \longrightarrow \mathbf{D}(A), \quad \mathsf{R}u_{\ast}: \mathbf{D}(A) \longrightarrow \mathbf{D}(B). $$
With the above notations, $u^{\ast}$ is called the \emph{inverse image} of $u$, while $\mathsf{R}u_{\ast}$ is called the \emph{direct cohomological image} of $u$. In particular, for every $\mathbf{U}$-small category $A$, the canonical functor $p_{A}: A \rightarrow e$ from $A$ to the point defines an adjunction:
  $$ (p_{A})^{\ast}: \mathbf{D}(e) \longrightarrow \mathbf{D}(A), \quad \mathsf{R}(p_{A})_{\ast}: \mathbf{D}(A) \longrightarrow \mathbf{D}(e). $$
For each object $M$ (resp. arrow $f$) in $\mathbf{D}(e)$, we denote by $M_{A}$ (resp. $f_{A}$ the associated object $(p_{A})^{\ast}(M)$ (resp. arrow $(p_{A})^{\ast}(f)$) in $\mathbf{D}(A)$. Moreover, given a functor $u: A \rightarrow B$ and an object $b \in \mathsf{Ob} B$, we denote by $F \vert_{A/b}$ the inverse images of objects $F$ in $\mathbf{D}(A)$ with respect to the projection functor $\zeta(u,b): A/b \longrightarrow A$. Finally, for each object of $\mathbf{D}(A)$, we define the notation
   $$ \mathsf{H}^{\ast}(A;F) =_{df} \mathsf{R}(p_{A})_{\ast}(F) $$
to indicate the \emph{total cohomology} of $A$ with coefficients in $F$. The cohomology groups of $A$ with coefficients in $F$ are the $\mathbf{U}$-abelian groups
  $$ H^{n}(A; F) =_{df} H^{n}(\mathsf{H}^{\ast}(A;F)). $$

\end{aforisma}

\begin{aforisma}
 Regarding each $\mathbf{U}$-abelian group $G$ as a complex concentrated in degree zero in $\mathsf{C}^{\ast}(\mathcal{A}b_{\mathbf{U}})$, i.e., a complex $C$ such that $C^{0}=G$ and $C^{n} =0$ for $n \neq 0$, there is a clear embedding of $\mathcal{A}b_{\mathbf{U}}$ into $\mathsf{C}^{\ast}(\mathcal{A}b_{\mathbf{U}})$, and hence, we can project any $\mathbf{U}$-abelian group into $\mathbf{D}(e) \simeq \mathsf{D}(\mathcal{A}b_{\mathbf{U}})$. Therefore, we define, for each $\mathbf{U}$-small category $A$, the abelian groups:
   $$ H^{n}(A) =_{df} H^{n}(A; \mathbb{Z}_{A}), \quad \quad n \in \mathbb{Z}, $$
which are the integer cohmology groups of $A$.

\end{aforisma}

\begin{proposition}\label{Derivator properties}
  \begin{enumerate}
    \item For every $\mathbf{U}$-small category $A$, the family of functors
        $$ a^{\ast}: \mathbf{D}(A) \longrightarrow \mathbf{D}(e), \quad a \in \mathsf{Ob} A,$$
     is jointly conservative, i.e., an arrow $f$ of $\mathbf{D}(A)$ is an isomorphism if, and only if, $f_{a}$ is an isomorphism in $\mathbf{D}(e)$ for all objects $a \in \mathsf{Ob} A$.
    \item For every morphism $u: A \rightarrow B$ between $\mathbf{U}$-small categories and every object $b \in \mathsf{Ob} B$, the canonical natural transformation
       $$ \mathsf{R}(p_{A/b})_{\ast} \zeta(u,b)^{\ast} \longrightarrow b^{\ast} \mathsf{R}u_{\ast} $$
induced from the $2$-square
\[
\xymatrix{
   A/b \ar[d]_{p_{A/b}} \ar[rr]^{\zeta(u,b)}  &   &  A \ar[d]^{u} \\
   e \ar[rr]_{b}   &   &  B
}
\]    
    is an isomorphism. In particular, for every coefficient $F$ in $\mathbf{D}(A)$, the canonical arrow
        $$ \mathsf{H}^{\ast}(A/b; F \vert_{A/b}) \longrightarrow (\mathsf{R}u_{\ast} F)_{b} $$
    is an isomorphism.
    \item For every morphism $u: A \rightarrow B$ between $\mathbf{U}$-small categories and every object $b \in \mathsf{Ob} B$, the canonical natural transformation
       $$ \mathsf{R}(u/b)_{\ast} (j_{A,b})^{\ast} \longrightarrow (j_{B,b})^{\ast} \mathsf{R}u_{\ast} $$
induced from the Cartesian $2$-square
\[
\xymatrix{
   A/b \ar[d]_{u/b} \ar[rr]^{j_{A,b}}  &   &  A \ar[d]^{u} \\
   B/b \ar[rr]_{j_{B,b}}   &   &  B
}
\]    
    is an isomorphism. 
    
  \end{enumerate}

\end{proposition}

\begin{definition}\label{Cohomological equivalence}
 A morphism $u: A \rightarrow B$ of $\mathbf{U}$-small categories is called a \emph{cohomological equivalence} if the canonical natural transformation
   $$ \mathsf{R}(p_{B})_{\ast}(p_{B})^{\ast} \longrightarrow (\mathsf{R}p_{A})_{\ast}(p_{A})^{\ast} $$
is an isomorphism.

\end{definition}

\begin{proposition}
 The localizing functor
    $$ \gamma: \mathcal{C}at_{\mathbf{U}} \longrightarrow (\mathcal{W}_{\mathbf{D}})^{-1} \mathcal{C}at_{\mathbf{U}} $$
sends an arrow $u$ of $\mathcal{C}at_{\mathbf{U}}$ to an isomorphism precisely when $u$ is a cohomological equivalence. In other words, the set of cohomological equivalences is strongly saturated

\end{proposition}

\begin{corollary}
  The set of cohomological equivalences in $\mathcal{C}at_{\mathbf{U}}$ is weakly saturated.

\end{corollary}

\begin{proposition}
  If $u,v: A \rightarrow B$ are two homotopic equivalent arrows in $\mathcal{C}at_{\mathbf{U}}$, then their inverse images $u^{\ast}$ and $v^{\ast}$ are isomorphic.

\end{proposition}

\begin{corollary}
 If two $\mathbf{U}$-small categories are homotopic equivalent, then they are cohomological equivalent. In particular, any two equivalent $\mathbf{U}$-small categories are cohomological equivalent.

\end{corollary}

\begin{proposition}
 If $A$ is a $\mathbf{U}$-small category with a terminal object, then the canonical arrow $p_{A}:A \rightarrow e$ from $A$ to the point is a cohomological equivalence.

\end{proposition}

\begin{lemma}
 Let 
 \[
 \xymatrix{
   A \ar[dr]_{p} \ar[rr]^{u}  &    &  B \ar[dl]^{q} \\
                              &  C &
 }
 \]
be a commutative triangle of $\mathbf{U}$-small categories. If $u$ is a local $\mathcal{W}$-equivalence over $C$, then the canonical natural transformation
   $$ \mathsf{R}q_{\ast} q^{\ast} \longrightarrow \mathsf{R}p_{\ast} p^{\ast} $$
is an isomorphism, i.e., $u$ is also a local cohomological equivalence over $C$.

\end{lemma}

\begin{proposition}
 The cohomological equivalences in $\mathcal{C}at_{\mathbf{U}}$ form a fundamental localizer according to (\ref{Fundamental localizers}). 
 
\end{proposition}

\begin{theorem}
 The functor
      $$ \mathcal{C}at_{\mathbf{U}} \longrightarrow \mathcal{A}b^{\mathbb{Z}}, \quad X \mapsto (H^{n}(X))_{n \in \mathbb{Z}}, $$
is a homotopy invariant according to (\ref{Homotopy invariant}).

\end{theorem}

\begin{aforisma}\label{Quasi-objects digression}
 Let $\mathscr{E}$ be a $\mathbf{U}$-category and $C$ be a $\mathbf{U}$-small full subcategory of $\mathscr{E}$. Then, one can form for each object $S$ in $\mathscr{E}$, the $\mathbf{U}$-small category $C/S$. If $S$ is an object of $C$, then $C/S$ admits a terminal object, and hence, $\mathsf{T}(C/S) \cong \mathsf{T}(e)$ for every homotopy invariant $\mathsf{T}$. In particular, $H^{0}(C/S) \cong H^{0}(e) \cong \mathbb{Z}$, and $H^{n}(C/S) \cong H^{n}(e) =0$ for every $n \neq 0$. However, the category $C/S$ could be aspherical even when $S$ is not an object of $C$. 

\end{aforisma}

\begin{example}\label{Z quasi-model}
 Let $\mathcal{C}omm$ be a skeleton of the category of \emph{finitely presentable commutative rings with unit} and define $\mathcal{A}ff =_{df} \mathcal{C}omm^{o}$. We denote formally by
  $$ Spec: \mathcal{C}omm \longrightarrow \mathcal{A}ff, \quad A \mapsto Spec(A), \quad \varphi \mapsto (^{a} \varphi) $$
the usual contravariant functor, sending an object to iteself, and an arrow to the corresponding dual arrow in $\mathcal{A}ff$. The category $\mathcal{A}ff$ is called the category of \emph{affine schemes}. We can choose a universe $\mathbf{U}$ for which the category $\mathcal{C}omm$ (and hence $\mathcal{A}ff$) is $\mathbf{U}$-small. Let $\mathcal{K}_{0}$ be the full subcategory of $\mathcal{A}ff$ formed by affine schemes of fields of characteristics zero. Then, the objects of $\mathcal{K}_{0}/Spec(\mathbb{Z})$ correspond to ring homomorphisms $\varphi: \mathbb{Z} \rightarrow k$ from $\mathbb{Z}$ to a field of characteristics zero. One can verify that $Spce(\mathbb{Z})$ is not an object of $\mathcal{K}_{0}$ (since $\mathbb{Z}$ is not even a field), but the category $\mathcal{K}_{0}/Spec(\mathbb{Z})$ is still aspherical. Indeed, since every ring homomorphism $\varphi: \mathbb{Z} \rightarrow k$ representing an object of $\mathcal{K}_{0}/Spec(\mathbb{Z})$ can be factorized by a unique field homomorphism of the form $\overline{\varphi}: \mathbb{Q} \rightarrow k$, the canonical projection $\pi:\mathbb{Z} \rightarrow \mathbb{Q}$ corresponds to a terminal object $^{a} \pi: Spec(\mathbb{Q}) \rightarrow Spec(\mathbb{Z})$ of $\mathcal{K}_{0}/Spec(\mathbb{Z})$, and hence, $\mathcal{K}_{0}/Spec(\mathbb{Z})$ is aspherical. Clearly, one could reproduce the previous arguments considering the categories $\mathcal{K}_{p}$ of affine schemes of fields of characteristics $p$ for any prime number $p$, since every ring homomorphism $\phi: \mathbb{Z} \rightarrow k$ from $\mathbb{Z}$ to a field of characteristics $p$ admits a unique factorization through a field extension $ \overline{\phi}:\mathbb{F}_{p} \rightarrow k$, where $\mathbb{F}_{p} = \mathbb{Z}/p\mathbb{Z}$, and hence, the canonical projection $\pi: \mathbb{Z} \rightarrow \mathbb{F}_{p}$ corresponds to a terminal object of $\mathcal{K}_{p}/Spec(\mathbb{Z})$. Now, let $\mathcal{K}$ be the full subcategory of $\mathcal{A}ff$ formed by affine schemes of fields. Then, the functor
   $$ \mathcal{K}/Spec(\mathbb{Z}) \longrightarrow \mathcal{K} $$
is locally aspherical over $\mathcal{K}$... Since locally asphericity implies weak homotopy equivalence, the functor   
 
\end{example}

\begin{definition}\label{Quasi-objects}
 With the notations of (\ref{Quasi-objects digression}), an object $S$ of $\mathscr{E}$ is called a \emph{quasi-object} of $C$ when the category $C/S$ is aspherical.

\end{definition}

\begin{remark}
 With the notations of (\ref{Quasi-objects digression}), there exists a functor
   $$ \mathscr{E} \longrightarrow \mathcal{C}at_{\mathbf{U}}, \quad S \mapsto C/S $$
which can be prolonged to a functor
  $$ C_{/?}: \mathscr{E} \longrightarrow \mathsf{Hot}_{\mathbf{U}}, \quad S \mapsto C/S $$
via localization. The above functor sends each object of $C$ to the point, but it also sends each quasi-object of $C$ to the point. Conversely, it follows from the fact that $\mathcal{W}_{\infty}$ is \emph{strongly saturated} that, if $C_{/?}$ sends an object $S$ of $\mathscr{E}$ to the point, then $S$ is a quasi-object of $C$. Therefore, the quasi-objects of $C$ are precisely the objects in $\mathscr{E}$ sent to the point through the functor $C_{/?}$, which means that they are homotopically indistinguishable from the ones of $C$.

\end{remark}

\begin{aforisma}\label{Euler char, Hot finite, Measure}
  Let $A$ be a $\mathbf{U}$-small category. The integer number
     $$ \chi(A) =_{df} \sum_{n \in \mathbb{Z}} (-1)^{n} dim H^{n}(A) $$
when it exists, is called the \emph{Euler characteristics} of $A$. Whenever the above sum does not exists as an integer number, we say that $\chi(A)$ is arbitrarily large or arbitrarily small. We are most interested in the cases when $H^{n}(A) =0$ expect for a finite number of indices, and the abelian groups $H^{n}(A)$ are all of finite dimension. More generally, a $\mathbf{U}$-small category $A$ is called \emph{homotopically finite} when $\chi(A)$ exists in $\mathbb{Z}$. Now, supposing the situation of (\ref{Quasi-objects digression}), we can compare the measures $\chi(C/S)$ and $\chi(C/S')$ for any two distinct objects $S$ and $S'$ of $\mathscr{E}$. If $\chi(C/S)$ is not defined, then $S$ is arbitrarily far away from $C$. Otherwise, the integer number $\chi(C/S)$ gives a measure of how much the object $S$ deviates from be an object of $C$. If $\|\chi(C/S)-1\| < \|\chi(C/S')-1\|$, then one could say that $S$ is homotopically closer to $C$ in comparison to $S'$. Here, we should take $\chi(C/S)-1$ because $\chi(e)=1$, and hence we need to make a correction in order to grant that, whenever $S$ is an object of $C$, then its deviation from $C$ is zero. However, the general idea is to use homotopy invariants $\mathsf{T}: \mathcal{C}at_{\mathbf{U}} \rightarrow \mathcal{A}$ as measures of deviations for $S$ to be an object of $C$. 

\end{aforisma}

\section{Co-Homological deviation of formulas}

\subsection*{Introduction}

In this work, we attempt to clarify two questions: \emph{what is a model} and \emph{what means to say that a model $M$ satisfies a sentence $\varphi$}? We show that if the theory of institutions answers the former, then  (co)homology necessarily answers the later. 

The idea of a cohomological interpretation of logic is inspired by Ren\'{e} Guitart, who even wrote the slogan: logic is homological algebra. In \cite{Gui3}, Guitart asserts the following: let $\mathcal{L}$ be a first-order language and let $\varphi$ be a formula in $\mathcal{L}$. Then, we can consider the full subcategory category $I$ of the category of $\mathcal{L}$-structures $\mathcal{C}$ generated by the models of $\varphi$. Given any $\mathcal{L}$-structure $X$, we can form the relative category $I/X$, where objects are pairs $(i,s)$ with $i \in Ob(I)$ and $s \in Hom_{\mathcal{C}}(i,X)$, and arrows $\alpha: (i,s) \rightarrow (j,t)$ are arrows $\alpha: i \rightarrow j$ in $I$ such that $s = t \circ \alpha$. Then, there exists an evident projection functor
      $$ \pi_{X}: I/X \longrightarrow I, \quad (i,s) \mapsto i, \quad \alpha \mapsto \alpha. $$
Using Andr\'{e}'s cohomology in \cite{Gui3}, Guitart constructs cohomology groups $H^{n}(I/X;T)$ for the data:
     $$ \mathcal{C}^{o} \supset I^{o} \xrightarrow{T} \mathcal{A}b ,$$
where $T: I^{o} \rightarrow \mathcal{A}b$ is some presehaf of abelian groups over $A$. These cohomology groups have the following intriguing property: if $X$ is an object of $I$, i.e., $X$ is a model for the formula $\varphi$, then $H^{n}(I/X;T)=TX$ for $n=0$ and $H^{n}(I/X;T)=0$ for $n \neq 0$. In particular, if we take $T$ as being the constant presehaf of an abelian group $\Lambda$, we have $H^{n}(I/X;\Lambda)= \Lambda$ for $n=0$ and $H^{n}(I/X;\Lambda)=0$ for $n \neq 0$. Hence, these cohomology groups $H^{n}(I/X;T)$ measure the deviation of an interpretation of the $\mathcal{L}$-structure to be a model for a formula $\varphi$. In other words, they measure the relation 
   $$X \models \varphi. $$
If we can choose a presheaf $T: I^{o} \rightarrow \mathcal{A}b$ such that $H^{n}(I/X;T) \neq 0$ for $n >0$, then the relation $X \models \varphi$ does not holds, which means that $X$ is not a model for $\varphi$. If $X$ is a model of some formal theory written in the language $\mathcal{L}$ and $\varphi$ is an arbitrary formula in $\mathcal{L}$ such that $H^{n}(I/X;T) \neq 0$ for $n >0$, then we already know (by completeness) that $\varphi$ can not be a theorem of our formal theory. Yet, we have more, because these cohomology groups $H^{n}(I/X;T)$ in some sense also measure the deviation of $\varphi$ from the formal theory.

We remark that the first of our questions does not ask `\emph{what is the natural or canonical model for some formal language?}', like Boolean (resp. Heyting) algebras for classical (resp. intuitionist) propositional logic and $\mathcal{L}$-structures for some first-order language $\mathcal{L}$, but what is a model independent of any chosen formal language and for every interpretant.

In this work, we generalize the results of Guitart, and show how deep is his insights in the following way:
\begin{enumerate}
  \item First, we consider not only the classical model theory of first-order logic, but an independent model theory, extensively developed by Diaconescu and called theory of institutions. In this theory, satisfiability is a primitive concept, and we have the general form of model theory for every notion of interpretant.
  \item Secondly, we not only construct a cohomology theory which measures the deviation of a model with respect to a formula, but the deviation for a whole theory.  
  \item Our cohomology theories will be deduced from the language of derivators, and hence, we are in a much more general picture. This point is crucial for it implies that the deviation of a model with respect to a formula is not a particularity of Andr\'{e}'s cohomology but is a cohomological phenomena in general. In particular, we can assign to the category $I/X$ in Guitart's construction a homotopy type. The projection functor $\pi_{X}$ from $I/X$ to $I$ proves a morphism of homotopy types, and if $X$ is an object of $I$, then this projection functor $\pi_{X}$ is a weak homotopy equivalence of small categories, which means that $I/X$ has the same homotopy type of $I$. 
    
\end{enumerate}

Ultimately, we have a homotopical interpretation of logic. Given an institution $\mathbf{I}$, a signature $\Sigma$ of $\mathbf{I}$, a complete theory $\Gamma$ of $\mathbf{I}$ and a sentence $\varphi$ in $\Gamma$, then $\varphi$ is a semantic consequence of $\Gamma$ iff the category $\mathsf{Mod}[\varphi]/M$ is contractible for every model $M$ of $\Gamma$.

\subsection*{Model theory}



Let  
     $\mathbf{I} = (\mathscr{S}ig, \mathsf{Fm}, \mathsf{Mod}, \models) $ be a  local institution 
     \[\xymatrix{
&\mathscr{S}ig\ar[ld]_{\mathsf{Mod}}\ar[rd]^{\mathsf{Fm}}&\\
(\mathcal{CAT}_{\mathsf{U}})^{op}&\models&\mathcal{E}ns_{\mathsf{U}}
}\]

\begin{definition}\label{Quasi-models}
 Let $\mathcal{W}$ be a fundamental localizer. A local $\Sigma$-structure $M$ is a $\mathcal{W}$-\emph{quasi-model} of $\Gamma$ when $M$ is a $\mathcal{W}$-quasi-object of $Mod(\phi)$ for every $\phi \in \Gamma$. Moreover, we say that $M$ is a \emph{quasi-model} of $\Gamma$ when it is $\mathcal{W}$-quasi-model of $\Gamma$ for every fundamental localizer $\mathcal{W}$.

\end{definition}

For each fundamental localizer $\mathcal{W}$ and each $\Sigma$-structure $M$, the notation $M \models_{\mathcal{W}} \Gamma$ indicates that $M$ is a  $\mathcal{W}$-quasi-model of $\Gamma$. Clearly, $M \models \Gamma$ always implies $M \models_{\mathcal{W}} \Gamma$, because whenever $M$ is a $\Gamma$-model, it is also an object of $Mod(\phi)$ for each $\phi \in \Gamma$. However, it is an easy consequence of (\ref{Z quasi-model}) that not every $\mathcal{W}$-quasi-model is a model, which means that $M \models_{\mathcal{W}} \Gamma$ \emph{does not necessarily implies} $M \models \Gamma$.

\begin{definition}
 Let $\mathcal{W}$ be a fundamental localizer. $\Gamma \Vdash_{\mathcal{W}} \phi$ if for any $\mathcal{W}$-quasi-model $M$ of $\Gamma$, $M$ is also a $\mathcal{W}$-quasi-model of $\phi$. The relation $\Gamma \Vdash_{\infty} \phi$ means that $\Gamma \Vdash_{\mathcal{W}} \phi$ for all fundamental localizers $\mathcal{W}$.

\end{definition}

In the case when $\Gamma$ is a set with just one formula $\phi$, to say that $M$ is a $\mathcal{W}$-quasi-model of $\phi$ is equivalent to say that $M$ is a quasi-object of $Mod(\phi)$. The crucial open question for us now is the following one: is there a syntactic counterpart $\vdash_{\mathcal{W}}$ of the relation $\Vdash_{\mathcal{W}}$? If so, then what is the notion of proof related to it?

\begin{proposition}
  If $M \models \phi$, then $H^{0}(Mod(\phi)/M) \cong \mathbb{Z}$ and $H^{n}(Mod(\phi)/M) =0$ for $n \neq 0$. In particular, if $\Gamma$ is a set of $\Sigma$-formulas and $H^{n}(Mod(\phi)/M) \neq 0$ for some formula $\phi \in \Gamma$ and some $n \neq 0$, then $M$ can not be a $\Gamma$-model.   
   
\end{proposition}

The previous proposition motivates the following

\begin{definition}\label{Errancy}
  A $\Sigma$-structure $M$ is an \emph{errancy} of $\Gamma$ when
     $$ \chi(Mod(\phi)/M) -1 \neq 0 $$
for at least one formula $\phi \in \Gamma$. The set of formulas $\phi \in \Gamma$ for which the above equation is valid is called the \emph{curvature} of $M$ in $\Gamma$.

\end{definition}

\begin{remark}
  The logic we are proposing in this work \emph{is not a logic of correctness}, but a  a \emph{logic of deviance}. The objects of interest of this logic are not models or even quasi-models, but precisely the deviances. It seems that even when correctness fails, there is a possible study of the error. The name curvature in the definition (\ref{Errancy}) has its inspirations from Gauss-Bonnet theorem, the one relating the local curvature and global Euler characteristics of suitable smooth manifolds.

\end{remark}

\subsection*{Proof systems}

In the following, we re-interpret proof systems in light of homotopical and cohomological methods.

\begin{definition}
  A propositional logic if a pair $(\mathsf{Fm}, \vdash)$ formed by a set $\mathsf{Fm}$ of formulas and a relation $\vdash$ between sets of formulas and formulas, satisfying the following conditions
  
  - (Triviality) If $\phi \in \Gamma$, then $\Gamma \vdash \phi$.
  
  - (Weakness) If $\Delta \subseteq \Gamma$ and $\Delta \vdash \phi$, then $\Gamma \vdash \phi$

  - (Transitivity) If $\Gamma \vdash \phi$ for all $\phi \in \Gamma'$ and $\Gamma' \vdash \mathsf{B}$, then $\Gamma \vdash \mathsf{B}$.
 
\end{definition}

For each set of formulas $\Gamma$, we define the set
  $$ \overline{\Gamma} = \{ \phi \in \mathsf{Fm}: \Gamma \vdash \phi\} $$
called the 	deductive close of $\Gamma$. It is immediate from the definitions that $\Gamma \subseteq \overline{\Gamma}$, and, if $\Gamma \subset \Gamma'$, then $\overline{\Gamma} \subset \overline{\Gamma}'$. A set of formulas $\Gamma$ is called a \emph{theory} when $\Gamma = \overline{\Gamma}$. The formulas in $\mathsf{Fm}$ form a category with exactly one arrow $\phi \rightarrow \mathsf{B}$ if $\phi \vdash \mathsf{B}$.

\begin{proposition}\label{Cohomological interpretation of theorems}
Let $\Gamma$ be a theory and $\phi$ be a formula. Equivalent conditions:
 \begin{enumerate}
   \item $\Gamma \vdash \phi$.
   \item $\Gamma/\phi$ is aspherical.
   \item $H^{0}(\Gamma/\phi) \cong \mathbb{Z}$ and $H^{n}(\Gamma/\phi) =0$ for $n \neq 0$.
   \item $\Gamma/\phi \neq \varnothing$.
   
 \end{enumerate}

\end{proposition}

\begin{proof}
  Suppose that $\Gamma$ is a theory. Hence, $\Gamma \vdash \phi$ iff $\phi \in \Gamma$. If $\Gamma \vdash \phi$, then $\phi$ is an object of $\Gamma$, which implies that $\Gamma/\phi$ admits a terminal object, and thus, it is aspherical, which proves that (1) implies (2). It is immediate that (2) implies (3) and (3) implies (4) (the last implication is clear through a counter-positive argument). The fact that (4) implies (1) follows by weakness.

\end{proof}

Thence, the non-theorems of a given theory $\Gamma$ are the formulas $\phi$ for which $H^{n}(\Gamma/\phi) \neq 0$ for some $n \neq 0$.

\begin{definition}
  A formula $\phi$ is called a $n$-\emph{curvature} of a set of formulas $\Gamma$ when $H^{i}(\Gamma/\phi) =0$ for all $i >n$. The $0$-curvatures of $\Gamma$ are precisely the theorems of $\Gamma$ (in virtue of (\ref{Cohomological interpretation of theorems})). An $\omega$-\emph{curvature} $\phi$ of $\Gamma$ is a formula such that, for every $n \in \mathbb{N}$, there exists $i > n$ for which $H^{i}(\Gamma/\phi) \neq 0$.

\end{definition}

In some sense, the family of abelian groups $(H^{n}(\Gamma/\phi))_{n \in \mathbb{Z}}$ could be called the abelian obstruction register of the relation $\Gamma \vdash \phi$. If two formulas $\phi$ and $\mathsf{B}$ have isomorphic abelian obstruction registers, then they have the same type of errancy from $\Gamma$. We use the notation $\Gamma \vdash_{k} \phi$, with $k \in \omega^{+}= \omega \cup \{\omega\}$, to indicate that $\phi$ is a $k$-curvature of $\Gamma$, and $\mathsf{Fm}(\Gamma;k)$ to indicate the set formed by these $k$-curvatures. Then, there is an hierarchy below:
    $$ \mathsf{Fm} = \bigsqcup_{k \in \omega^{+}} \mathsf{Fm}(\Gamma; k). $$

\subsection*{Some calculations}





\subsubsection*{Cohomology of finite products}

Let $\lambda_j: \lim_J X_j\to X_j$, $\mu^{j}_\alpha: T{s_n}\to \sum_{\alpha\in S_n(X_j)}Ts_n$ e $\lambda_\alpha: T{s_n}\to \sum_{\alpha\in S_n(\lim_J X_j)}Ts_n$ os (co)cones limits (here $\alpha:s_n\to\cdots\to X_j$). Then
\[\xymatrix{
  T{s_n} \ar[r]_{\lambda_\alpha} \ar[rd]_{\mu^{j}_{s_0(\alpha \pi_j)}}  &  \quad \sum_{\alpha\in S_n(\lim_J X_j)}T{s_n} \ar@{-->}^{k_j}[d]\\
                                                                          &   \sum_{\alpha\in S_n(X_j)}Ts_n
}\] By the universal property for $f:i\to j\in J $ we have $f^*k_i=k_j$ since 
\[\xymatrix{
  T{s_n} \ar[r]_{\mu^i_\alpha} \ar[rd]_{\mu^{j}_{s_0(\alpha f)}}  &  \quad \sum_{\alpha\in S_n( X_i)}T{s_n} \ar@{-->}^{f^*}[d]\\
                                                                          &   \sum_{\alpha\in S_n(X_j)}Ts_n
}\] and then $k:\sum_{\alpha\in S_n(\lim_J X_j)}T{s_n}\to \lim_J\sum_{\alpha\in S_n(X_j)}Ts_n$. Moreover, $\lim_J S_n X_j\cong S_n(\lim_J X_j) $ since $C/\lim_J X_j =\lim_J C/X_j$ e $N$ commutes with limits. If $J$ is finite and discrete, then the inverse of $k$ is given by
\[\xymatrix{
  \sum_J\sum_{\alpha\in S_n(X_j)}Ts_n \ar[d] \ar@{..>}[r]  &  \quad \sum_{\alpha\in\lim_J S_n( X_j)}T{s_n} \ar[d]^{\pi_\alpha}\\
    \sum_{\alpha\in S_n(X_j)}Ts_n\ar[r]_{\pi_{\alpha_j}}\ar@{-->}[ur]&    Ts_n
}\]
Therefore \[C_n(\prod_J X_j,T)\cong \sum_J C_n(X_j,T)\]
and  
$$H_n (\prod_m X_i, T) \cong \sum_m H_n(X_i,T).$$

\subsubsection*{Cohomology of filtered colimits}

\newcommand{\colim}{\text{colim}}

Let $\mu_j:S^n(X_j)\to\colim_J S^n(X_j)$. Define \[k_j:\sum_{\alpha\in S^n(X_j)}Ts_n\to \sum_{\alpha\in\colim_J S^n X_j}Ts_n\] by $k_j(x)=(x_{s_0(\alpha\mu_j)})_\alpha$. Consider the epimorphism \[k:\sum_J \sum_{\alpha\in S^n(X_j)}Ts_n\to \sum_{\alpha\in\colim_J S^n X_j}Ts_n\]
Note that $k(j,a)=k(i,b)$ iff
\[(a_{s_0(\alpha\mu_j)})_\alpha=(b_{s_0(\alpha\mu_i)})_\alpha\quad\quad(\star)\]
Given $f:i\to k$ note $ (a_{s_0(\alpha\mu_i)})_{\alpha\in \colim_J S^n(X_j)}=(a_{s_0(s_0(\alpha \mu_k)f^*)})_{\alpha\in \colim_J S^n(X_j)}$. Since $J$ is a filtered category  $(\star)$ is equivalent to
\[\exists f:i\to k\exists g:j\to k [(a_{s_0(\alpha f)})_{\alpha\in S^n(k)}=(b_{s_0(\alpha g)})_{\alpha\in S^n(k)}]
\]
Since $\Delta^n$ is compact, the nerve construction commutes with filtered colimits: this isomorphism guarantee that $x_k\to s_0\to ...\to s_n$ came from $\colim_J x_k\to s_0\to...\to s_n$. Thus $\mu_i$ is injective;  $\colim N(X_j/C)_n=N(\colim( X_j/C))_n=N((\colim X_j)/C)_n$. Since the equivalence relation above is the relation that describes a filtered colimit as a quotient of a coproduct, the isomorphim theorem gives us
\[\colim_J \sum_{\alpha\in S^n(X_j)}Ts_n\cong \sum_{\alpha\in\colim_J S^n X_j}Ts_n\]
That is: $C^n(\colim_J X,T) \cong \colim_JC(X_j,T)$. Since  filtered colimits commutes with:  finite limits, arbitrary colim and images (= coker ker), this entails 

$$H^n(\colim_J X_j,T) \cong \colim_J H^n(X_j,T)$$

\subsubsection*{A concrete calculation}










Consider the inclusion functor $Field\hookrightarrow Ring$ and the functor $(-)^{\times}:Field\to Ab$. We have 
\[d^0:\sum_{\mathbb{Z}\to F}F^{\times}\to\sum_{\mathbb{Z}\to F\to K}K^{\times} \]
given by $(d^0(x))_{f,F}=x_{dom(f)}-x_{cod(f)}$ and
\[d^1:\sum_{\mathbb{Z}\to F\to K}K^{\times}\to\sum_{\mathbb{Z}\to F\to K\to T}T^{\times} \]
given by $\mathbb{Z}\to F\xrightarrow{f}K\xrightarrow{g}T$, and $(d^{1}(x))_{g,f}=x_{g,K}-x_{gf,F}+x_{f,F}$. Thus
\[H^1(\mathbb{Z},(-)^\times) = \frac{ \{x_{f,F}: \forall\mathbb{Z}\to F\xrightarrow{f}K\xrightarrow{g}T(x_{g,K}-x_{gf,F}+x_{f,F}=0) \} }{\{x_{f,F}:\exists y\in\sum_{\mathbb{Z}\to F} F^\times (x_{f,F}=y_{dom(f)}-y_{cod(f)})\}}\]
Given $(x_{f,F})$ set $y_F= x_{\chi_F, \mathbb{Q}}$ where $\chi_F:\mathbb{Q}\to F$ is the  unique morphism when $char F=0$ and $y_F= x_{\chi_F,\mathbb{F}_p}$ otherwise. Given a sequence $\mathbb{Z}\to \mathbb{Q}\to F\xrightarrow{g}K$ we have
\[x_{g,F}-x_{g\chi_F,\mathbb{Q}}+x_{\chi_F,\mathbb{Q}}=0\]
\[\implies x_{g,F}-x_{\chi_K,\mathbb{Q}}+x_{\chi_F,\mathbb{Q}}=0\implies x_{g,F}=y_K-y_F\]
and similarly for $\mathbb{Z}\to \mathbb{F}_p\to F\xrightarrow{g}K$. Therefore $H^1=0$, however $\mathbb{Z}$ is not a field!!!

\section{Final remarks and future developments}


\begin{enumerate}
    \item A natural scenario is consider the Andre's  cohomology  concerning the  inclusion functor $j: Mod[T] \hookrightarrow L-Str$, where $L$ is a finitary language and $T$ is  a set of $\Sigma$-sentences in the first-order finitary language  $L_{\omega,\omega}$. It makes sense chose $X$,  a general $L$-structure, and consider the cases where $M_i \in Mod[T]$ and take the functor ${\cal T}$ as the functor "free abelian group generated by..."  and consider:

The deviation of $X$ be in $Mod[T]$ represented by cohomologies given by: 
$$\alpha : M_n \to M_{n-1} \to \cdots \to M_1 \to M_0 \to X$$
 a  sequence of $L$-homomorphisms (or $L$-elementary embeddings or even $L$-pure embeddings),

or 

The deviation of $X$ be in $Mod[T]$ represented by homologies given by: 
$$\beta : X \to M_0 \to M_1\to \cdots \to M_{n-1}\to M_{n}$$
 a  sequence of $L$-homomorphisms (or $L$-elementary embeddings or even $L$-pure embeddings).

\item  It may be illuminating study the following subclasses of $L-Str$ formed by all the structures $X$ satisfying some of the conditions below:

(a) $H^n(X,j) = 0$, $n \geq 1$ (this contains the  class $Mod[T]$)

(b) $H^n(X,j) = 0$, $n \geq n_0$ for some chosen $n_0 \geq 1$.

(c) The reunion of the subclasses described in item (b) above, for all $n_0 \geq 1$.

(d) The complement of the subclass (c), i.e. formed by the structures $X$ such that the set of $n \in \mathbb{N}$ such that $H^n(X,j) \neq 0$ is  infinite.

$ $

\item It could be interesting "recalculate and compare" the groups of (co)homology of a general structure $X$ when we move from $Mod[T]$ to  an intermediary subcategory $S$: 

$Mod(T) \hookrightarrow S \hookrightarrow \Sigma-str$

where $S$ is in some of the cases:

(a) $S$ is the closure of $Mod[T]$ under constructions like substructures and (monomorphic) filtered colimits --since these are related with well-known model theoretic results (for instance, $S=Mod(T_\forall) =$ the class of substructures of some model of $T$).

(b) $S$ is the class of consequences of $T$ in a  language that contains $L_{\omega, \omega}$: for instance $L_{\omega_1, \omega}$ is an extension of $L_{\omega, \omega}$ that has a well behaved model theory 
and many mathematical notions that  can not be described by $L_{\omega, \omega}$-sentences, can be  characterized by  $L_{\omega_1, \omega}$-sentences (e.g., well-founded posets, divisible groups, etc...).

\end{enumerate}




\end{document}